\newif\ifpictures
\picturestrue

\documentclass[12pt]{amsart}
\usepackage{amssymb,amsmath}
 \usepackage{amsopn}
 \usepackage{xspace}
 \usepackage[dvips]{graphicx}

\numberwithin{equation}{section}
\newtheorem{thm}{Theorem}

\newtheorem{lemma}[thm]{Lemma}
\newtheorem{cor}[thm]{Corollary}
\newtheorem{conj}[thm]{Conjecture}

\numberwithin{thm}{section}

\newcounter{FNC}[page]
\def\newfootnote#1{{\addtocounter{FNC}{2}$^\fnsymbol{FNC}$%
     \let\thefootnote\relax\footnotetext{$^\fnsymbol{FNC}$#1}}}

\newcommand{\C}{\mathbb{C}}

\newcommand{\Q}{\mathbb{Q}}
\newcommand{\R}{\mathbb{R}}

\newcommand\cV{{\ensuremath{\mathcal{V}}}\xspace}

\newcommand{\lam}{\lambda}

\newcommand{\Sig}{\Sigma}

\newcommand{\lf}{\left}
\newcommand{\ri}{\right}
\newcommand{\ra}{\rightarrow}

\newcommand{\Lera}{\Leftrightarrow}

\DeclareMathOperator{\bs}{\backslash} 

\DeclareMathOperator{\rank}{rank}

\title[Separating inequalities for nonnegative polynomials that are not SOS]{Separating inequalities for nonnegative polynomials that are not sums of squares}
\author{Sadik Iliman} \author{Timo de Wolff}

\address{Goethe-Universit\"at, FB 12 -- Institut f\"ur Mathematik,
Postfach 11 19 32, D-60054 Frankfurt am Main, Germany}

\email{\{iliman,wolff\}@math.uni-frankfurt.de}

\thanks{Second author is supported by DFG grant TH 1333/2-1.}


\begin{document}

\begin{abstract}
Ternary sextics and quaternary quartics are the smallest cases where there exist nonnegative polynomials that are not sums of squares (SOS). A complete classification of the difference between these cones was given by G. Blekherman via analyzing the extreme rays of the corresponding dual cones. However, an exact computational approach in order to build separating extreme rays for nonnegative polynomials that are not sums of squares is a widely open problem. We provide a method substantially simplifying this computation for certain classes of polynomials on the boundary of the PSD cones. In particular, our method yields separating extreme rays for every nonnegative ternary sextic with at least seven zeros. As an application to further instances, we compute a rational certificate proving that the Motzkin polynomial is not SOS. 
\end{abstract}

\maketitle

\section{Introduction}
\label{Sec:Introduction}

We consider real polynomials in the vector space of all homogeneous polynomials in $n$ variables of degree $d$, denoted by $H_{n,d}$. For every $p \in H_{n,d}$ we denote its real projective variety as $\cV(p)$. Let $P_{n,d} \subset H_{n,d}$ be the cone of all nonnegative polynomials in $n$ variables of degree $d$.

Inside $H_{n,2d}$, there are two full dimensional convex cones of special interest, the cone of nonnegative polynomials and the cone of sums of squares (for a general background about nonnegative polynomials and sums of squares see e.g. \cite{Delzell:Prestel:Survey,Lasserre:Moments,Marshall:Survey,Reznick:Survey}; for some metric and convexity properties of these cones see \cite{Blekherman:Convexity}).
\begin{eqnarray*}
	P_{n,2d} 	& := & \lf\{f\in H_{n,2d} \ : \  f(x)\geq 0, \text{ for all } x\in\mathbb R^n \ri\}\\
	\Sig_{n,2d} 	& := & \lf\{f\in P_{n,2d} \ : \  f=\sum\nolimits_{i}^{}f_i^2 \text{ for some } f_i\in H_{n,d}\ri\}.
 \end{eqnarray*}

The investigation of the relationship between the cone of nonnegative polynomials and the cone of sums of squares began in the seminal work of Hilbert when he showed that the cone of nonnegative polynomials coincides with the cone of sums of squares exactly in the cases of bivariate forms $(n=2)$, quadratic forms $(2d=2)$ and ternary quartics $(n=3, 2d=4)$ (\cite{Hilbert:Seminal}).

The Motzkin polynomial  $m(x,y,z) = x^4y^2+x^2y^4-3x^2y^2z^2+z^6$ was the first explicitly known example for a nonnegative polynomial which is not a sum of squares. Most proofs for this fact are based on term by term inspections (see e.g. \cite{Motzkin:Selected,Reznick:Survey}). In near past other proofs were found, e.g. using representation theory (see \cite{Bosse:Motzkin}). 

In \cite{Blekherman:Volume} Blekherman showed that for fixed dimension $2d \geq 4$ there are significantly more nonnegative polynomials than sums of squares as $n$ tends to infinity. However, the question of precisely when nonnegative polynomials begin to significantly overtake sums of squares is much less understood. In the smallest cases where there exist nonnegative polynomials which are not sums of squares $((n,2d)=(3,6), (4,4))$ the general conjecture is that these two cones do not differ very much. This conjecture is supported by the following two facts: 
Firstly, the maximal dimensional difference between exposed faces of the cone of nonnegative polynomials and sums of squares is one (see \cite{Blekherman:Dimension}). Secondly, all extreme rays of the dual sums of squares cone $\Sig_{3,6}^{*}$ have rank one or rank seven (see \cite{Blekherman:PSDandSOS}). 

Recently, in \cite{Blekherman:et:al} it is shown that except the discriminant there is a unique component of the algebraic boundary of $\Sig_{3,6}^{}$ with degree 83200 which indicates the complicated structure of the SOS cone. But still, the geometry and the relationship between these two cones in the smallest cases are less understood.

In the smallest cases $(n,2d)=(3,6)$ and $(n,2d)=(4,4)$ Blekherman showed that it is precisely the Cayley-Bacharach relation that prevents sums of squares from filling out the cone of nonnegative polynomials. More precisely, in \cite{Blekherman:PSDandSOS} it is shown that every separating extreme ray in the dual SOS cone for a given nonnegative polynomial that is not a sum of squares depends on a 9-point configuration for $(n,2d) = (3,6)$ resp. a 8-point configuration for $(n,2d) = (4,4)$ coming from the intersection of two qubic resp. three quadric polynomials. Furthermore, given an appropriate 9-point (resp. 8-point) configuration, one can write down an extreme ray of the dual SOS cone (see Theorems \ref{Thm:BlekhermanMain} and \ref{Thm:BlekhermanExtremeRay}) corresponding to faces of maximal dimension of the SOS cone.

A central problem in this area is how to obtain the separating inequalities efficiently. This can always be done in a numerical way (see Section \ref{Subsec:Blekhermanresults}), but is widely open for exact methods currently. Hence, finding constructive methods for computing these inequalities is one main research issue. Blekherman's result does not provide an efficient symbolic way to obtain a proper 9-point (resp. 8-point) configuration to solve this problem (see \ref{Subsec:Blekhermanresults} for further details). 

The key idea of this article is to construct a proper 9-point (resp. 8-point) configuration out of a given initial set of points. Specifically, we investigate nonnegative polynomials $p$ which lie on the boundary of the cones $P_{3,6}$ and $P_{4,4}$ (which cover most of the explicitly known nonnegative polynomials that are not SOS, see \cite{Blekherman:Volume}). Our main result, Theorem \ref{Thm:Main}, provides a sufficient condition for using $k$ zeros of $p$ as a subset of a 9-point (resp. 8-point) configuration. The idea is to fill up the set of $k$ zeros with $9-k$ (resp. $8-k$) points such that a genericity and a quadratic condition based on the Cayley-Bacharach relation holds (note that $k \leq 10$ for $p \in P_{3,6} \setminus \Sig_{3,6}$ and $p \in P_{4,4} \setminus \Sig_{4,4}$; see \cite{Blekherman:et:al,Choi:Lam:Reznick:RealZeros}). Given these conditions, which are computationally easily checkable (at least for $(n,2d) = (3,6)$), we can construct a separating extreme ray immediately. This method reduces the complexity of constructing separating extreme rays via symbolic computation significantly. Furthermore, it yields rational certificates for rational point configurations and even for rational varieties $\cV(p) \subset \Q^3$ in case of $(n,2d) = (3,6)$.

We show that for $p \in P_{3,6} \setminus \Sig_{3,6}$ and $k = 7$ (resp. $p \in P_{4,4} \setminus \Sig_{4,4}$ and $k = 6$) almost every 9-point (resp. 8-point) configuration containing the seven (resp. six) zeros leads to a certificate for a nonnegative polynomial $p$ to be not SOS.

In Section \ref{Sec:Prelim} we review some curve theoretical issues as, e.g., the Cayley-Bacharach relation and present Blekherman's results on the SOS cones $\Sig_{3,6}$ and $\Sig_{4,4}$. In Section \ref{Sec:Relaxation} we state and prove our main Theorems \ref{Thm:Main} and \ref{Thm:Main2} for ternary sextics and quaternary quartics and discuss exactness and rationality of our methods. Section \ref{Sec:7Points} deals with the special case of polynomials with exactly seven zeros. We show that in this case our method generically yields a separating extreme ray (Theorem \ref{Thm:7PointCase}). Finally we discuss the difficulties of dropping zeros in our method in Section \ref{Sec:Motzkin} via applying it to the Motzkin polynomial and constructing a rational extreme ray certificate for it. To the best of our knowledge, our certificate is the first separating rational extreme ray certificate for the Motzkin polynomial. Furthermore, based on the Motzkin polynomial we investigate some geometric aspects of the set of appropriate point configurations for our method (see Figure \ref{Fig:Motzkin}). In the appendix we discuss an example of the seven point case and an example of a point configuration for the Motzkin polynomial where our method does not yield a certificate.

\section{Preliminaries}
\label{Sec:Prelim}

\subsection{Curve theoretical background}
\label{Subsec:CurveBackground}
We recall some classical results from algebraic geometry. From now on we consider every investigated polynomial to be homogeneous. We start with the Cayley-Bacharach relation. It exists in various formulations (see \cite{Eisenbud:Green:Harris}); we use the ones given in \cite{Blekherman:PSDandSOS}.

\begin{lemma}
Let $p_1,p_2 \in \C[x_1,x_2,x_3]$ be two ternary cubics intersecting transversely in nine projective points $\gamma_1,\ldots,\gamma_9$. Let $v_1,\ldots,v_9$ be affine representatives of $\gamma_i$. Then there is a unique linear relation on the values of any ternary cubic on $v_i$
\begin{eqnarray}
	\sum_{j = 1}^9 u_j f(v_j) = 0  \text{ for all ternary cubics } f \label{Equ:CB}
\end{eqnarray}
with nonzero $u_j \in \C$. Furthermore, if \eqref{Equ:CB} is satisfied, then the following genericity condition holds
\begin{eqnarray}
	\text{no four of the } v_i \text{ lie on a line and no seven on a quadric.} \label{Equ:GenCond}
\end{eqnarray}
\label{Lem:CB}
\end{lemma}

\vspace{-0.5cm}

Analogously, we have for quaternary quadrics

\begin{lemma}
Let $p_1,p_2,p_3 \in \C[x_1,\ldots,x_4]$ be three quaternary quadrics intersecting transversely in eight projective points $\gamma_1,\ldots,\gamma_8$. Let $v_1,\ldots,v_8$ be affine representatives of $\gamma_i$. Then there is a unique linear relation on the values of any quaternary quadric  on $v_i$
\begin{eqnarray}
	\sum_{j = 1}^8 u_j f(v_j) = 0  \text{ for all quaternary quadrics } f \label{Equ:CB2}
\end{eqnarray}
with nonzero $u_j \in \C$. Furthermore, if \eqref{Equ:CB2} is satisfied, then the following genericity condition holds
\begin{eqnarray}
	\text{no five of the } v_i \text{ lie on a plane (i.e. projective linear 2-space).} \label{Equ:GenCond2}
\end{eqnarray}
\label{Lem:CB2}
\end{lemma}

\vspace{-0.5cm}

For the genericity condition \eqref{Equ:GenCond2} see e.g. \cite{Hilbert:Seminal}. Note that, if all points $v_j$ are real, then all Cayley-Bacharach coefficients $u_j$ are real, too (see e.g. \cite[Lemma 4.1]{Blekherman:PSDandSOS}) and can be computed by solving a system of linear equations with the coefficients of forms in $H_{3,3}$ resp. $H_{4,2}$ as variables.

Each of the conditions \eqref{Equ:GenCond} and \eqref{Equ:GenCond2} can be checked easily by investigating the minors of the matrix given by the vectors $v_j$.\\

For the description of separating extreme rays of $\Sig_{3,6}^*$ yielded by Blekherman's Theorem \ref{Thm:BlekhermanMain} one needs to investigate 9-point configurations given by intersecting two ternary cubics. These ternary cubics need to be coprime. The following lemma shows that this generically is the case (see \cite{Reznick:PSDandSOS}).

\begin{lemma}
Suppose $A := \{v_1,\ldots,v_8\}$ is a set of eight distinct points in $\R^3$, no four on a line and no seven on a quadric and let $f_1,f_2$ be a basis for the vector space of all homogeneous cubics with projective variety affinely represented by $A$. Then $f_1$ and $f_2$ are relatively prime.
\label{Lem:ReznickBezout}
\end{lemma}

This lemma yields that one can apply Bezout's theorem in order to compute a ninth intersection point $v_9$ of $f_1$ and $f_2$. However, $v_9$ might not be different from $v_1,\ldots,v_8$ (i.e. the intersection multiplicity might be greater than 1). But, again, generically this will not be the case as the following lemma shows (see \cite{Nie:Discriminants}).

\begin{lemma}
Let $f_1,f_2$ be two homogeneous polynomials in $n$ variables (with $n \geq 2$) of degree $d$ and generic coefficients. The discriminant $\Delta(f_1,f_2)$ vanishes if and only if $f_1(x) = f_2(x) = 0$ has a singular solution. The set of polynomials for which this is the case is a hypersurface.
\label{Lem:NieDiscriminant}
\end{lemma}

In Section \ref{Sec:7Points} we will investigate the special case of polynomials $p \in P_{3,6} \setminus \Sig_{3,6}$ with exactly seven zeros. 
In this context we use the following lemma (see \cite{Reznick:PSDandSOS}).

\begin{lemma}
Suppose $A$ is a set of seven distinct points in $\R^3$, no four on a line and no seven on a quadric with basis $f_1,f_2,f_3$ for the vector space of homogeneous cubics with projective variety affinely represented by $A$. Then $f_1,f_2,f_3$ have no common zeros outside of $A$.
\label{Lem:ReznickSeven}
\end{lemma}

\subsection{Blekherman's results}
\label{Subsec:Blekhermanresults}

In \cite{Blekherman:PSDandSOS} Blekherman was able to fully characterize the extreme rays of the dual SOS cones $\Sig_{3,6}^*$ and $\Sig_{4,4}^*$, basically via using the Cayley-Bacharach relation. We recall his main result, concentrating on ternary sextics.

\begin{thm}
Suppose that a ternary sextic $p$ is nonnegative but not a sum of squares. Then there exist two real cubics $q_1, q_2\in H_{3,3}$ intersecting transversely in $9$ projective points $\gamma_1,\dots,\gamma_9$ which yield a certificate for $p \in P_{3,6} \setminus \Sig_{3,6}^{}$. More precisely, let $v_1,\dots,v_9$ be affine representatives of $\gamma_1,\dots,\gamma_9$. Then there exists a linear functional $l: H_{3,6} \ra \R$ given by 
$$l(f) \ = \ \sum_{i=1}^{9}a_i f(v_i)$$ 
for some $a_i\in\mathbb C$ such that $l(p)<0$ and $l(\Sig_{3,6}^{})\geq 0$. Furthermore, at most two of the points $\gamma_i$ are complex.
\label{Thm:BlekhermanMain}
\end{thm}

This theorem works the same way for $p \in P_{4,4} \setminus \Sig_{4,4}$ (see \cite{Blekherman:PSDandSOS}).\\

Recall that for every $l \in \Sig_{3,6}^*$ there is a corresponding quadratic for $Q_l$ defined by $Q_l: H_{3,3} \ra \R, f \mapsto l(f^2)$ (see e.g. \cite{Blekherman:PSDandSOS, Laurent:Survey}). One defines the rank of a linear functional $l \in \Sig_{3,6}^*$ by $\rank(l) := \rank(Q_l)$. In \cite{Blekherman:PSDandSOS} it is shown that every extreme ray of $\Sig_{3,6}^{*}$, which does not correspond to point evaluation (i.e., a rank $1$ form), is given by a rank $7$ form which comes from a $9$-point evaluation. 

\begin{thm}
Suppose $l$ spans an extreme ray of $\Sig_{3,6}^{*}$ which does not correspond to a point evaluation. Let $W_l$ be the kernel of the
corresponding quadratic form $Q_l$ and suppose $q_1,q_2\in W_l$ intersect transversely in $9$ real projective points $\gamma_1,\dots,\gamma_9$ with affine representatives $v_1,\dots,v_9$ such that the unique Cayley-Bacharach relation is given by 
\begin{eqnarray*}
	  u_1f(v_1)+\dots +u_9f(v_9)=0 \quad for \, f\in H_{3,3}. 
\end{eqnarray*}
Then $Q_l$ can be uniquely written as 
$$Q_l(f) \ = \ a_1f(v_1)^2+\dots +a_9f(v_9)^2$$ 
with exactly one single negative coefficient $a_k$ and the rest of the $a_i$ strictly positive
and $a_k$ is given by 
\begin{eqnarray}
	  a_k \ = \ \frac{-u_k^2}{\frac{u_1^2}{a_1}+\dots+\frac{u_9^2}{a_9}-\frac{u_k^2}{a_k}}. \label{Equ:a_9}
\end{eqnarray}
Furthermore, any such form is extreme in $\Sig_{3,6}^{*}$.
\label{Thm:BlekhermanExtremeRay}
\end{thm}

Again, an analogue version for $p \in P_{4,4} \setminus \Sig_{4,4}$ holds (see \cite{Blekherman:PSDandSOS}).\\

Suppose $p \in P_{3,6} \setminus \Sig_{3,6}$ and we want to construct a separating extreme ray $l$ for $p$ using the upper theorem. Therefore, we need to find two coprime ternary cubics $q_1,q_2$ intersecting in 9 points. But, $q_1,q_2$ need to be contained in the kernel $W_l$ of the quadratic form $Q_l$ corresponding to $l$. Hence, one already needs to know $l$ to determine $q_1,q_2$.

This problem can be avoided by choosing a 9-point configuration $A = \{v_1,\ldots,v_9\}$ coming from an intersection of some real ternary cubics $q_1,q_2$. So, a separating extreme ray $l$ is obtained by finding an appropriate $a = (a_1,\ldots,a_9)$ satisfying \eqref{Equ:a_9} with respect to $A$ such that $l_a(p) < 0$. The question whether such an $a$ exists is unclear a priori though it can be answered by quantifier elimination methods (see e.g. \cite{Basu:Pollack:Roy:Algo,Bochnak:Coste:Roy:RealAlgebraic}). But, to the best of our knowledge, no methods are known to compute an appropriate $a$ in a symbolic, exact way efficiently.

However, one can solve this problem numerically. Let $p\in P_{3,6} \setminus \Sig_{3,6}^{}$ be a ternary sextic and $r\in \text{int}\left(\Sig_{3,6}^{}\right)$ (e.g. $r=x^6+y^6+z^6$ or $r=(x^2+y^2+z^2)^3$). Consider the following semidefinite optimization problem.
\begin{eqnarray*}
	  & & \min_{\lambda\in\mathbb R} \lambda \qquad	  \text{such that} \quad p+\lambda r\in \Sig_{3,6}^{}
\end{eqnarray*}

For $\lambda$ minimal the polynomial  $p+\lambda r$ is strictly positive and lies on the boundary of $\Sig_{3,6}$. Hence, $p+\lambda r$ is a sum of exactly three squares $s_1^2+s_2^2+s_3^2$ (see \cite{Blekherman:PSDandSOS}).

The polynomials $s_1,s_2,s_3$ have no common zeros and an appropriate linear combination of two of these polynomials can be used as $q_1$ and $q_2$ in Blekherman's theorem. Of course, the computation of the corresponding nine intersection points will be difficult and not exact, too. Furthermore, getting ``nice'' values (e.g. a rational minimal $\lambda$) depends also highly on the choice of the polynomial $r\in \text{int}\left(\Sig_{3,6}^{}\right)$. It is not clear how to choose $r$ in dependence of $p$.

In the case $p\in P_{4,4} \setminus \Sig_{4,4}$ this approach works the same way. For $\lambda$ minimal the polynomial $p+\lambda r$ is a sum of exactly four squares $p+\lambda r = s_1^2+s_2^2+s_3^2 + s_4^2$. Three of these four $s_j$ have a common zero (see \cite{Blekherman:PSDandSOS}).

\section{A Certificate for boundary polynomials not to be SOS}
\label{Sec:Relaxation}

Our approach to construct a separating extreme ray for a given boundary polynomial $p \in \partial P_{3,6} \setminus \Sig_{3,6}$ is to investigate certain point sets $A := \{v_1,\ldots,v_9\}$ containing the variety $\cV(p)$, satisfying the genericity condition \eqref{Equ:GenCond} and for which we can certify that there are coprime polynomials $q_1,q_2 \in H_{3,3}$ with $\cV(q_1) \cap \cV(q_2) = A$.

Note that, if we talk about zeros of homogeneous polynomials in this and the following sections, then we always consider their affine representatives with slight abuse of notation.

The easiest case is when $p$ has at least eight zeros $v_1,\ldots,v_8$ (satisfying \eqref{Equ:GenCond}). Lemma \ref{Lem:ReznickBezout} provides the existence of coprime $q_1,q_2$ vanishing on $v_1,\ldots,v_8$ and thus a ninth point $v_9$ is given by Bezout's Theorem. For a generic set of zeros $v_1,\ldots,v_8$ the corresponding coprime polynomials $q_1,q_2$ have generic coefficients and hence, due to Lemma \ref{Lem:NieDiscriminant}, we have $v_9 \notin \{v_1,\ldots,v_8\}$ generically. Thus, $A := \{v_1,\ldots,v_9\}$ satisfies \eqref{Equ:GenCond} generically. This yields a certificate $l$ immediately since for any choice of $a_1,\ldots,a_8$ we obtain an $a_9 < 0$ by \eqref{Equ:a_9} and hence
\begin{eqnarray*}
	l(p) \ = \ \sum_{j = 1}^9 a_j p(v_j) \ = \ a_9 p(v_9) \ < \ 0.
\end{eqnarray*}

In the following we generalize this idea to any number of zeros between one and eight. We choose the zeros $v_1,\ldots,v_k$ of a polynomial $p \in \partial P_{3,6} \bs \Sig_{3,6}$ as a subset of the nine intersection points $A := \{v_1,\ldots,v_9\}$ of two coprime ternary cubics. We provide a symbolic method based on genericity conditions which yields a separating extreme ray if one finds a $(9-k)$-point configuration satisfying some quadratic relation. Specifically, the following theorem holds.

\begin{thm}
Let $p \in \partial P_{3,6}$ such that $p$ is not $SOS$. Let $A := \{v_1,\ldots,v_9\} \subset \R^3$ such that the genericity condition \eqref{Equ:GenCond} holds and $\cV(p) = \{v_1,\ldots,v_k\}$ with $1 \leq k \leq 8$  (i.e.: the projectivization of $A$ is the intersection of two coprime polynomials $q_1,q_2 \in H_{3,3}$). Then there exists a certificate $l_a: H_{3,6} \ra \R, f \mapsto \sum_{j = 1}^9 a_j f(v_j)$, $a := (a_1,\ldots,a_9) \in \R^9$ with respect to $A$ for $p$ to be not $SOS$, if the following inequality holds
\begin{eqnarray}
	(u_{k+1}^2 + \cdots + u_8^2) (p(v_{k+1}) + \cdots + p(v_8))	& <	& u_9^2 p(v_9). \label{Equ:OurRelaxation}
\end{eqnarray}
Here the $u_j$ are given by the unique Cayley-Bacharach relation on $A$ and $l_a$ is an extreme ray of $\Sig_{3,6}^*$.
\label{Thm:Main}
\end{thm}

Note that the Cayley-Bacharach coefficients $u_j$ can be computed by solving a system of linear equations (see exemplarily in Section \ref{Sec:Motzkin}). Additionally, all $u_j$ are rational, if every point in $A$ is rational. Note furthermore that, for an arbitrary $p$, it is not clear whether an $A$ with $\cV(p) \subset A$ satisfying \eqref{Equ:OurRelaxation} does always exist. We discuss certain special cases in the two following sections.

\begin{proof}
Let $p \in \partial P_{3,6}$ with $\cV(p) = \{v_1,\ldots,v_k\}$, $1 \leq k \leq 8$ such that the genericity condition \eqref{Equ:GenCond} holds for $\cV(p)$. We choose points $v_{k+1},\ldots,v_8$ such that \eqref{Equ:GenCond} is still satisfied. We obtain $v_9$ as the intersection of two relatively prime, cubic polynomials spanning up the vector space of all ternary cubics vanishing on $v_1,\ldots,v_8$ (see Lemma \ref{Lem:ReznickBezout}). Notice that we obtain $v_9 \notin \{v_1,\ldots,v_8\}$ generically due to Lemma \ref{Lem:NieDiscriminant} and $v_9$ has to be real since $v_1,\ldots,v_9$ is the intersection of two real polynomials (see e.g. \cite{Reznick:PSDandSOS}).
Let $u_1,\ldots,u_9$ be the unique Cayley-Bacharach coefficients for $v_1,\ldots,v_9$ in the sense of \eqref{Equ:CB}. Since $v_1,\ldots,v_9 \in \R^3$ we have $u_1,\ldots,u_9 \in \R$ (see \cite[Lemma 4.1]{Blekherman:PSDandSOS}).

By Theorem \ref{Thm:BlekhermanExtremeRay} every vector $a := (a_1,\ldots,a_9) \in \R^9$ satisfying \eqref{Equ:a_9} with $a_1,\ldots,a_8 > 0$, $a_9 < 0$ yields an extreme ray $l_a : H_{3,6} \ra \R, f \mapsto \sum_{j = 1}^9 a_j f(v_j)$ of the dual SOS cone $\Sig_{3,6}^*$. $l_a$ is the dual of a separating hyperplane for $p$ if $l_a(p) < 0$, i.e., since $\cV(p) = \{v_1,\ldots,v_k\}$, if $a_{k+1} p(v_{k+1}) + \cdots + a_9 p(v_9) < 0$. By \eqref{Equ:a_9} this is equivalent to
\begin{eqnarray*}
	a_{k+1} p(v_{k+1}) + \cdots - \frac{u_9^2}{\frac{u_1^2}{a_1}+\dots+\frac{u_8^2}{a_8}} p(v_9)	& <	& 0 \\
	\Lera (a_{k+1} p(v_{k+1}) + \cdots + a_8 p(v_8)) \cdot \lf(\frac{u_1^2}{a_1}+\dots+\frac{u_8^2}{a_8}\ri)	& <	&  u_9^2 p(v_9).	
\end{eqnarray*}
Let $\lam_{a_1,\ldots,a_k} := \sum_{j = 1}^k \frac{u_j^2}{a_j} \cdot (a_{k+1} p(v_{k+1}) + \cdots + a_8 p(v_8)) > 0$. Thus, $l_a(p) < 0$, if 
\begin{eqnarray*}
	\lam_{a_1,\ldots,a_k} + \sum_{j = k+1}^8 p(v_j) \lf(u_j^2 + \sum_{i \in \{k+1,\ldots,8\} \bs \{j\}} \frac{a_j u_i^2}{a_i}\ri)	& <	& u_9^2 p(v_9).
\end{eqnarray*}
We choose $a_{k+1} := 1,\ldots,a_8 := 1$ and obtain
\begin{eqnarray*}
	\lam_{a_1,\ldots,a_k} + (u_{k+1}^2 + \cdots + u_8^2) (p(v_{k+1}) + \cdots + p(v_8))	& <	& u_9^2 p(v_9).
\end{eqnarray*}
Since $\lim_{a_1,\ldots,a_k \ra \infty} \lam_{a_1,\ldots,a_k} \searrow 0$ the relaxation \eqref{Equ:OurRelaxation} yields an extreme ray $l_a$ on $A$ separating $p$ from the SOS cone $\Sig_{3,6}$.

If on the other hand any SOS polynomial $g \in \Sig_{3,6}$ with $\cV(g) = \cV(p)$ would satisfy \eqref{Equ:OurRelaxation}, then it follows from the upper construction that $l_a(g) < 0$ for $a_1,\ldots,a_k$ sufficiently large, $a_{k+1},\ldots,a_8 = 1$ and $a_9$ given by \eqref{Equ:a_9}. This is a contradiction with Blekherman's Theorems \ref{Thm:BlekhermanMain} and \ref{Thm:BlekhermanExtremeRay}. Thus, \eqref{Equ:OurRelaxation} indeed is a certificate for $p$ to be not SOS.
\end{proof}

Note that it is also easy to show that condition \eqref{Equ:OurRelaxation} is never satisfied for SOS polynomials $g \in \Sig_{3,6}$ with $\cV(g) = \cV(p)$ in order to prove extremality of $l_a$ by direct calculation and without usage of Blekherman's Theorems. This follows already from $g = \sum_{j = 1}^r h_i^2$ and the Cayley-Bacharach relations.\\

In order to prove an analogon of the upper theorem for $P_{4,4} \setminus \Sig_{4,4}$, we need to show that Lemma \ref{Lem:ReznickBezout} also holds for a seven point set $A := \{v_1,\ldots,v_7\} \subset \R^4$. Generically, the vector space of all quadrics vanishing on $A$ has dimension three (see \cite{Eisenbud:Syzygie}).

\begin{lemma}
Suppose $A := \{v_1,\ldots,v_7\}$ is a set of seven distinct points in $\R^4$, no four on a plane such that $q_1,q_2,q_3$ is a basis for the vector space of all homogeneous quadrics with projective variety affinely represented by $A$. Then $q_1,q_2,q_3$ are relatively prime.
\label{Lem:ReznickAnalogon}
\end{lemma}

\begin{proof}
Suppose $q_1,q_2,q_3$ have a common factor $g$. Then $q_j = g \cdot q_j'$ for $j \in \{1,2,3\}$ and $g,q_j'$ have to be linear in $\R[x_1,\ldots,x_4]$. Due to the genericity condition at most three zeros (w.l.o.g. $v_1,v_2,v_3$) are located on $\cV(g)$ since otherwise there would exist at least five points are contained in a plane. Hence, $\cV(q_1'),\cV(q_2')$ and $\cV(q_3')$ share four points which is a contradiction since for each $j$ all points in $\cV(q_j')$ are contained in a line.
\end{proof}

\begin{thm}
Let $p \in \partial P_{4,4}$ such that $p$ is not $SOS$. Let $A := \{v_1,\ldots,v_8\} \subset \R^4$ such that the genericity condition \eqref{Equ:GenCond2}  holds and $\cV(p) = \{v_1,\ldots,v_k\}$ with $1 \leq k \leq 7$  (i.e.: the projectivization of $A$ is the intersection of three coprime polynomials $q_1,q_2,q_3 \in H_{4,2}$). Then there exists a certificate $l_a: H_{4,4} \ra \R, f \mapsto \sum_{j = 1}^8 a_j f(v_j)$, $a := (a_1,\ldots,a_8) \in \R^8$ with respect to $A$ for $p$ to be not $SOS$, if the following inequality holds
\begin{eqnarray}
	(u_{k+1}^2 + \cdots + u_7^2) (p(v_{k+1}) + \cdots + p(v_7))	& <	& u_8^2 p(v_8). \label{Equ:OurRelaxation2}
\end{eqnarray}
Here the $u_j$ are given by the unique Cayley-Bacharach relation on $A$ and $l_a$ is an extreme ray of $\Sig_{4,4}^*$.
\label{Thm:Main2}
\end{thm}

The proof works the same way as for Theorem \ref{Thm:Main} with the obvious modifications.

In fact, the proof of Theorem \ref{Thm:Main} already shows one possible way how to choose $a = (a_1,\ldots,a_9) \in \R^9$ to obtain a separating extreme ray $l_a$.

\begin{cor}
For $p \in \partial P_{3,6} \setminus \Sig_{3,6}$ and $A = \{v_1,\ldots,v_9\} \supset \cV(p)$ with \eqref{Equ:OurRelaxation} satisfied, one valid certificate is given by $a_1 = \cdots = a_k = N \in \R$ (for $N$ sufficiently large), $a_{k+1} = \cdots = a_8 = 1$ and $a_9$ given by \eqref{Equ:a_9}. For $p \in \partial P_{4,4} \setminus \Sig_{4,4}$ and $A = \{v_1,\ldots,v_8\} \supset \cV(p)$ with \eqref{Equ:OurRelaxation2} satisfied, one valid certificate is given by $a_1 = \cdots = a_k = N \in \R$ (for $N$ sufficiently large), $a_{k+1} = \cdots = a_7 = 1$ and $a_8$ given by analogon of \eqref{Equ:a_9} for $\Sig_{4,4}^*$ (see \cite{Blekherman:PSDandSOS}). 

In particular, $l_a$ is a rational certificate, i.e. every $a_j$ is rational, if every point $v_j \in A$ is rational.
\label{Cor:ExplicitRay}
\end{cor}

\vspace*{-0.8cm}

\begin{flushright}
$\Box$
\end{flushright}

If one is interested in computing rational certificates, then there is the following problem from an application point of view. Suppose, we have a rational variety $\cV(p) = \{v_1,\ldots,v_k\}$ and we choose $v_{k+1},\ldots,v_8 \in \Q^3$ (resp. $v_{k+1},\ldots,v_7 \in \Q^4$) such that the genericity condition \eqref{Equ:GenCond} (resp. \eqref{Equ:GenCond2}) holds (which is always possible). Then it is not clear a priori that the ninth intersection point $v_9 \in \R^3$ (resp. eighth intersection point $v_8 \in \R^4$), given by Bezout, is rational, too.

According to a recent preprint by Q. Ren \cite{Ren:CB}, for $p \in P_{3,6} \bs \Sig_{3,6}$, the ninth intersection point can always be computed exactly. In particular, it can be followed that $v_9$ will always be rational whenever $v_1,\ldots,v_8$ are rational and hence, whenever $\cV(p)$ is rational.

\begin{cor}
Let $p \in \partial P_{3,6} \setminus \Sig_{3,6}$ with $\cV(p) = \{v_1,\ldots,v_k\} \subset \Q^3$ and $\{v_{k+1},\ldots,v_8\} \subset \Q^3$ such that \eqref{Equ:GenCond} holds. Then there is a rational certificate $l_a$ on $A = \{v_1,\ldots,v_9\}$ with $v_9$ given by Bezout, whenever \eqref{Equ:OurRelaxation} holds.
\label{Cor:ExplicitRationalRay}
\end{cor}

\vspace*{-0.8cm}

\begin{flushright}
$\Box$
\end{flushright}

Note that in our Theorems \ref{Thm:Main} and \ref{Thm:Main2} we only consider real points $v_1,\ldots,v_9$ whereas in Blekherman's Theorem \ref{Thm:BlekhermanMain} (at most) one pair of complex conjugated points are allowed. Anyhow, Blekherman conjectures in \cite{Blekherman:PSDandSOS} that all extreme rays of the dual cones $\Sig_{3,6}^*$ and $\Sig_{4,4}^*$ can be described by nine (resp. eight) real points. Based on our results, we formulate a slightly different conjecture here.

\begin{conj}
For $p \in P_{3,6} \setminus \Sig_{3,6}$ there exist $v_1,\ldots,v_9 \in \R^3$ yielding a separating extreme ray for $p$ in the sense of Theorem \ref{Thm:BlekhermanMain}. Analogously, for $p \in P_{4,4} \setminus \Sig_{4,4}$.
\label{Conj:9realpoints}
\end{conj}

Clearly, Blekhermans conjecture implies ours, since if every extreme ray is real representable, every nonnegative polynomial that is not SOS can be separated by a real intersection. But it is unclear whether the two variants are equivalent or not and, if not, whether our conjecture is a strong relaxation of Blekherman's.

\section{The seven point case}
\label{Sec:7Points}

Let $p \in P_{3,6} \setminus \Sig_{3,6}$ and assume we are interested in finding a separating extreme ray $l_a$ in $\Sig_{3,6}^*$ for $p$. That means we need to find a generic 9-point set $A$ being the intersection of the varieties of two coprime polynomials $q_1,q_2 \in H_{3,3}$ such that the conditions in Theorem \ref{Thm:BlekhermanMain} hold. If $p$ is located on the boundary of $P_{3,6}$, i.e. $\cV(p) = \{v_1,\ldots,v_k\} \neq \emptyset$, then Theorem \ref{Thm:Main} and Corollary \ref{Cor:ExplicitRay} yield a certificate $l_a$ whenever one can fill up $\cV(p)$ with points $v_{k+1},\ldots,v_9$ such that condition \eqref{Equ:OurRelaxation} is satisfied.

Anyhow, it is not obvious a priori if resp. how points $v_{k+1},\ldots,v_9$ can be chosen such that the sufficient condition \eqref{Equ:OurRelaxation} of Theorem \ref{Thm:Main} holds (note that $v_9$ is always given by Bezout in this approach). It turns out that for $k = 7$, i.e. the easiest non-trivial case, the problem of choosing an appropriate $v_8$ is easy since almost every $v_8$ yields a $v_9$ such that \eqref{Equ:OurRelaxation} is satisfied. Note that for $p \in \partial P_{3,6} \setminus \Sig_{3,6}$ the variety $\cV(p)$ always satisfies the genericity condition \eqref{Equ:GenCond} (see \cite{Reznick:PSDandSOS}).

One reason why this case might be of special interest is that $k = 7$ with $v_1,\ldots,v_7$ satisfying the genericity condition \eqref{Equ:GenCond} is the smallest number of zeros of a nonnegative polynomial such that the dimensional difference between exposed faces of $P_{3,6}$ and $\Sig_{3,6}$ given by the vanishing of polynomials on these zeros is strictly positive (see \cite{Blekherman:Dimension}).

Theorem \ref{Thm:7PointCase} holds analogously in the case of $p \in P_{4,4} \setminus \Sig_{4,4}$ and $k = 6$ with the obvious modifications. We omit to formulate the result for this case separately.

\begin{thm}
Let $p \in \partial P_{3,6} \setminus \Sig_{3,6}$. Let $A := \{v_1,\ldots,v_9\} \subset \R^3$ such that the genericity condition \eqref{Equ:GenCond} holds and $\cV(p) = \{v_1,\ldots,v_7\}$. Then there exists a certificate $l_a: H_{3,6} \ra \R, f \mapsto \sum_{j = 1}^9 a_j f(v_j)$, $a_j \in \R$ with respect to $A$ for $p$ to be not $SOS$ if 
\begin{eqnarray}
      u_8^2 p(v_8) & \neq	& u_9^2 p(v_9). \label{Equ:7PointEquality}
\end{eqnarray}
Furthermore, $l_a$ is an extreme ray of $\Sig_{3,6}^*$.
\label{Thm:7PointCase}
\end{thm}

\begin{proof}
We choose $v_8$ such that the genericity condition \eqref{Equ:GenCond} still holds for $\{v_1,\ldots,v_8\}$ and obtain a real $v_9 \notin \{v_1,\ldots,v_8\}$ such that \eqref{Equ:GenCond} and Cayley-Bacharach holds for $A$ generically (see proof of Theorem \ref{Thm:Main}). Since all conditions of Theorem \ref{Thm:Main} are satisfied, there is a certificate $l$ if $u_8^2 p(v_8) < u_9^2 p(v_9)$. But since $\{v_1,\ldots,v_7,v_9\}$ also yields $A$ with Bezout's theorem, this condition holds w.l.o.g. as long as $u_8^2 p(v_8) \neq u_9^2 p(v_9)$.
\end{proof}

\begin{cor}
Conjecture \ref{Conj:9realpoints} holds for $p \in \partial P_{3,6} \setminus \Sig_{3,6}$ with $\#\cV(p) \geq 7$ and $p \in \partial P_{4,4} \setminus \Sig_{4,4}$ with $\#\cV(p) \geq 6$.
\end{cor}

Let $p \in P_{3,6}$ such that $p$ is not $SOS$ with $\cV(p) = \{v_1,\ldots,v_k\} \subset \R^3$. Assume that we want to test for a large amount of point sets $\{v_{k+1},\ldots,v_8\}$ in $\R^3$ if \eqref{Equ:OurRelaxation} can be satisfied for each point set (under the condition that \eqref{Equ:GenCond} holds for $v_1,\ldots,v_8$). We have to do two steps. Firstly, we have to compute $v_9$ by calculating two coprime polynomials $q_1,q_2$ vanishing on $\{v_1,\ldots,v_8\}$ via solving a system of linear equations on the coefficients of a ternary cubic and then calculate their Gr\"obner basis. If we are in the case $(n,2d) = (3,6)$, then this can also be done by using the formula given in \cite{Ren:CB}. Secondly, we have to compute the Cayley-Bacharach relation on $A := \{v_1,\ldots,v_9\}$ via solving a system of linear equations again.

In the seven point case, i.e. if we want to check \eqref{Equ:7PointEquality} for a large amount of point sets, we can avoid calculating the Cayley-Bacharach relation for each $\{v_8,v_9\}$ if we are not interested in computing the separating hyperplane given by $l_a$ explicitly. The trick is to use Reznick's Lemma \ref{Lem:ReznickSeven}. Let $V(\cV(p)) := \{h \in H_{3,3} : \cV(p) \subseteq \cV(h)\}$ be the vector space of all real ternary cubics vanishing on $\cV(p)$. By Lemma \ref{Lem:ReznickSeven} there is a basis $\{h_1,h_2,h_3\}$ of $V(\cV(p))$ such that $\cV(h_1) \cap \cV(h_2) \cap \cV(h_3) = \cV(p)$ and $h_1,h_2,h_3$ do not depend on the choice of $v_8$ and $v_9$.

\begin{cor}
Let $p$ and $A$ be as in Theorem \ref{Thm:7PointCase}. Let $h \in V(\cV(p))$ with $\cV(h) = \cV(p)$. Then there exists a certificate $l_a: H_{3,6} \ra \R, f \mapsto \sum_{j = 1}^9 a_j f(v_j)$, $a_j \in \R$ with respect to $A$ for $p$ to be not $SOS$ if 
\begin{eqnarray*}
      h^2(v_9) p(v_8) \ \neq \ h^2(v_8) p(v_9).
\end{eqnarray*}
Furthermore, $l_a$ is an extreme ray of $\Sig_{3,6}^*$.
\label{Cor:7PointCaseEasyCond}
\end{cor}

\begin{proof}
By Lemma \ref{Lem:ReznickSeven} there is a $h \in V(\cV(p))$ with $\cV(h) = \cV(p)$. Since $A$ satisfies \eqref{Equ:GenCond} and $h \in H_{3,3}$ the Cayley-Bacharach relation \eqref{Equ:CB} holds for $h$. Since $\cV(h) = \cV(p) = \{v_1,\ldots,v_7\}$ this means $u_8 h(v_8) + u_9 h(v_9) = 0$ and thus $u_8^2 h^2(v_8) = u_9^2 h^2(v_9)$. Hence, 
\begin{eqnarray*}
      u_8^2 p(v_8) \ \neq \ u_9^2 p(v_9)	& \Lera 	& h^2(v_9) p(v_8) \ \neq \ h^2(v_8) p(v_9).
\end{eqnarray*}
The assertion follows with Theorem \ref{Thm:7PointCase}.
\end{proof}

If we want to compute a specific $l_a$ the approach of Corollary \ref{Cor:7PointCaseEasyCond} does not suffice since by \eqref{Equ:a_9} we need to know the coefficients $u_j$ of the Cayley-Bacharach relation \eqref{Equ:CB} on $A$ to compute the scalar $a_9$ for $l_a$.\\

In the appendix we present an example of our method applied to a nonnegative polynomial with exactly seven zeros that is not SOS.

\section{An application: the Motzkin polynomial}
\label{Sec:Motzkin}
In this section we demonstrate applications of our method. It turns out that finding a separating extreme ray becomes more difficult for polynomials in $p \in \partial P_{3,6} \setminus \Sig_{3,6}$ with six or less zeros. The condition \eqref{Equ:OurRelaxation} provides more degrees of freedom and, in particular, the left hand side of this inequality has more than one summand. This fact yields that the set of point configurations $A := \{v_1,\ldots,v_9\} \subset \R^3$ with $\cV(p) \subset A$ which do not satisfy \eqref{Equ:OurRelaxation} will not be a lower dimensional subset in general -- in contrast to the seven point case (independently from which point corresponds to the negative entry $a_9$ in an extreme ray). The same difficulties arise for  $p \in P_{4,4} \setminus \Sig_{4,4}$ with five or less zeros.\\

Exemplarily, we investigate the Motzkin polynomial
\begin{eqnarray*}
	m(x,y,z) & =	&x^4y^2+x^2y^4-3x^2y^2z^2+z^6.
\end{eqnarray*}

The Motzkin polynomial has six zeros
$$
\begin{array}{lll}
	v_1 \ := \ (1,0,0), & v_2 \ := \ (0,1,0), & v_3 \ := \ (1,1,1), \\
	v_4 \ := \ (-1,1,1), &  v_5 \ := \ (1,-1,1), &  v_6 \ := \ (1,1,-1). \\
\end{array}
$$

As a first instance we choose $v_7 := (0,4,1)$ and $v_8 := (4,0,1)$. These eight points satisfy the genericity condition \eqref{Equ:GenCond}. This can be checked by looking at the $3 \times 3$ minors of the $8 \times 3$ matrix given by the coordinates $(x,y,z)$ of the points $v_1,\ldots,v_8$ and looking at the $6 \times 6$ minors of the $8 \times 6$ matrix given by the coordinates $(x^2,y^2,z^2,xy,xz,yz)$ of the points $v_1,\ldots,v_8$. Hence we can compute two coprime ternary cubics
\begin{eqnarray*}
	q_1 & =	 & -16 z^3+15 x^2 z+y^2 z+56 y^2 x-56 z^2 x \\
	q_2 & =	 & -4 z^3-x^2 y+4 x^2 z+15 y^2 x-15 z^2 x+z^2 y
\end{eqnarray*}
vanishing on $v_1,\ldots,v_8$ by solving the system of linear equations
\begin{eqnarray*}
	h(v_1) \ = \ 0, \ldots, h(v_8) \ = \ 0
\end{eqnarray*}
on the coefficients of $h$, where  
\begin{eqnarray*}
	  h	& :=	& b_1 x^3 + b_2 y^3 + b_3 z^3 + b_4 x^2 y + b_5 x^2 z + b_6 y^2 z + b_7 y^2 x + b_8 z^2 x + b_9 z^2 y + b_{10} x y z
\end{eqnarray*}
again. We compute the Gr\"obner basis
{\small
\begin{eqnarray*}
	& & \{7 z-26 z^2-15 z^3+26 z^4+8 z^5, -2 z+8 z^2+2 z^3-105 y-8 z^4+105 z^2 y,\\ 
	& & 8 z^4-422 z^3+1575 y^2-1583 z^2+422 z, x-1\}
\end{eqnarray*}
}
of $q_1$, $q_2$ and $x - 1$ with lexicographic ordering. We obtain $v_9 = (1, 1, -7/2)$ and compute the Cayley-Bacharach coefficients $u_j$ by solving the system of linear equations
\begin{eqnarray*}
	u_1 h(v_1) + \cdots + u_9 h(v_9) & =	& 0
\end{eqnarray*}
in $u_1,\ldots,u_9$. The solution is (up to multiplicity)
\begin{eqnarray*}
	  u 	& = & \lf(-64,-64, -\frac{40}{9}, -4, -4, \frac{24}{5}, 1, 1, \frac{118098}{5}\ri)^T
\end{eqnarray*}

We have
\begin{eqnarray*}
	(u_7^2 + u_8^2) \cdot (m(v_7) + m(v_8))	& = & 4, \\
	u_9^2 m(v_9) & =	& 228.
\end{eqnarray*}

Hence, the condition \eqref{Equ:OurRelaxation} of Theorem \ref{Thm:Main} is satisfied and we find a separating hyperplane for $m$ on $A$. According to Corollary \ref{Cor:ExplicitRay} we choose $a_1,\ldots,a_6 := 100$ and $a_7,a_8 := 1$. By \eqref{Equ:a_9} we obtain
\begin{eqnarray*}
	a_9 \ = \ \frac{-u_9^2}{\frac{u_1^2}{a_1}+\dots+\frac{u_8^2}{a_8}} \ = \ \frac{-14121476824050}{2143157}.
\end{eqnarray*}
We check the correctness of our result by
\begin{eqnarray*}
	l_a(m)	& =	& a_1 m(v_1) + \cdots + a_9 m(v_9) \\
		& =	& m(0,4,1) + m(4,0,1) - \frac{14121476824050}{2143157} \cdot m\lf(1,1,-\frac{7}{2}\ri) \\
		& =	& -\frac{1484936}{2143157} \ < \ 0.
\end{eqnarray*}
Thus, by Blekherman's Theorem \ref{Thm:BlekhermanExtremeRay}, $l_a$ is a (rational) extreme ray of $\Sig_{3,6}^*$ separating the Motzkin polynomial $m$ from the SOS cone.\\

In contrast to the seven point case, not every generic point configuration yields a separating certificate. For example, with the same approach it is easy to show that the instance $v_7 := (2/7,2/3,1)$ and $v_8 := (2/3,2/7,1)$ does not satisfy the condition \eqref{Equ:OurRelaxation} (see Appendix \ref{App:Motzkin}).

We show that for a symmetric choice of $v_7, v_8$, i.e. $v_7 = (q,s,1), v_8 = (s,q,1)$ with $q,s \in \R$, the set
\begin{eqnarray*}
	S	& :=	& \{(q,s) \in \R^2 \ : \ (u_7^2 + u_8^2) \cdot (m(v_7) + m(v_8)) < u_9^2 m(v_9)\}
\end{eqnarray*}
yieding a 9-point configuration which satisfies \eqref{Equ:OurRelaxation} is full dimensional with some nice geometric structure (see Figure \eqref{Fig:Motzkin}).

With the formula in \cite{Ren:CB} we obtain the ninth Cayley-Bacharach point
{\small\begin{eqnarray*}
	v_9	& =	& 
	\frac{1}{n(q,s)} \cdot \lf(\begin{array}{c}
	q^3s+2q^2s^2-q^2+s^3q-2sq-s^2 \\
	q^3s+2q^2s^2-q^2+s^3q-2sq-s^2 \\
	 q^3+q^2s-2q+s^2q-2s+s^3 \\	
	\end{array}\ri),
\end{eqnarray*}
}
where
{\small\begin{eqnarray*}
	n(q,s)	& = & 12 \cdot (q^3s^2-2q^3s+q^3+s^3q^2-2q^2s^2+q^2s-3q+4sq-2s^3q+s^2q+2-3s+s^3)
\end{eqnarray*}}
vanishing at $q = -2 - s, q = 1, s = 1$. Furthermore we obtain a non-generic point set for $q = s, q = -s, q = -1, s = -1$ and $q = 2 - s$. We compute the Cayley-Bacharach coefficients in dependence of $q,s$ and obtain
{\SMALL\begin{eqnarray*}
	(u_7^2 + u_8^2) \cdot (m(v_7) + m(v_8)) & = & \frac{64(1+s^4 q^2+q^4 s^2-3 q^2 s^2)}{(q - s)^4} \\
	u_9^2 m(v_9) & = & \frac{16 (2 q^4 s^2+q^4+4 s^3 q^3-4 q^3 s+2 s^4 q^2-6 q^2 s^2-2 q^2-4 s^3 q+4 s q+s^4-2 s^2+4)}{(q-s)^4}.
\end{eqnarray*}
}
Note that the numerator of $(u_7^2 + u_8^2) \cdot (m(v_7) + m(v_8))$ is exactly the dehomogenized Motzkin polynomial in $s,q$ for $z = 1$. We set
\begin{eqnarray*}
	K(q,s) / L(q,s)	
	& :=	& (u_7^2 + u_8^2) \cdot (m(v_7) + m(v_8)) - u_9^2 m(v_9) \\
	& = & 16(2q^2s^2-q^2+2sq-s^2+2)/(q-s)^2.
\end{eqnarray*}

Since $q \neq s$ by assumption we just need to investigate $K(q,s)$. Thus, by \eqref{Equ:OurRelaxation} we have $(q,s) \in S$ if and only if $K(q,s) < 0$ and $q \notin \{\pm s,\pm 1, \pm 2 - s\}$. Equivalently, $S = \emptyset$ if and only if $K(q,s)$ is PSD and $q \notin \{\pm s,\pm 1, \pm 2 - s\}$. Since $K(q,s)$ is a bivariate polynomial of degree 4, $K(q,s)$ is PSD if only if it is SOS. It can be checked easily that this is not the case. 

We provide a plot of the set $S := \{(q,s) : K(q,s) < 0\} \setminus \{(q,s) : q = 2 - s \text{ or } q = -2 -s\}$ in Figure \ref{Fig:Motzkin} (note that the other non-generic cases are not part of $S$ although they are relevant for the computation). Obviously, this set is symmetric in $q = s$ and $q = -s$, semi-algebraic and for every $q$ there is an $s$ such that $(q,s) \in S$.

\ifpictures
\begin{figure}[ht]
	\includegraphics[width=0.4\linewidth]{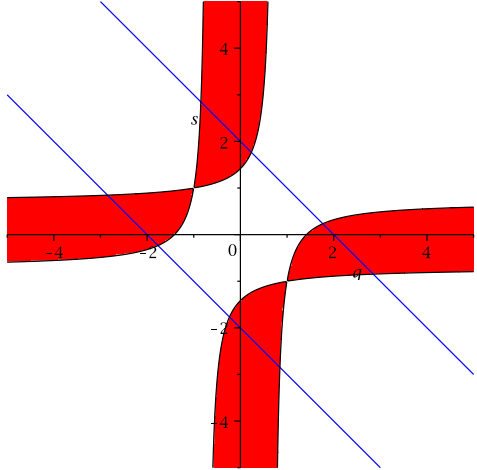}
	\caption{The set $S := \{(q,s) : K(q,s) < 0\} \setminus \{(q,s) : q = 2 - s \text{ or } q = -2 -s\}$ is given by the red area without the blue lines.}
	\label{Fig:Motzkin}
\end{figure}
\fi

Due to the rich geometric structure of $S$, it would be of interest to investigate the geometric structure of the set of appropriate point configurations satisfying \eqref{Equ:OurRelaxation} for general nonnegative polynomials with $k$ zeros.\\

In contrast, we briefly demonstrate the numerical method for finding a 9-point certificate given in Section \ref{Subsec:Blekhermanresults} and the corresponding problems. 
Let $r = (x^2 + y^2 +z^2)^3 \in \text{int}(\Sigma_{3,6})$. Then, consider the following semidefinite optimization problem
$$\min_{\lambda\in\mathbb R}^{} \lambda \qquad \text{such that}\quad m + \lambda (x^2+y^2+z^2)^3 \in \Sigma_{3,6}.$$
The optimal lambda is given numerically by $\lambda\approx 0.004596411406567$ and the corresponding sum of squares decomposition of 
$m + \lambda r \approx s_1^2 + s_2^2 + s_3^2$ is given by

{\small
$$m+\lambda (x^2+y^2+z^2)^3 \ \approx \ (-0.858558xz^2+0.941354xy^2+0.067796x^3)^2 + $$
	$$(-0.858558yz^2+0.067796y^3+0.941354x^2y)^2 + (-1.002295z^3+0.360838y^2z+0.360838x^2z)^2.$$
}

Now, one has to choose an appropriate linear combination of two of the polynomials $s_i$ to obtain the nine intersection points. But it seems unclear how to do this. E.g., since the coefficients of the $s_i$ are given numerically, computing the Gr\"obner basis of two of the $s_i$ and, say, $x-1$ cannot be expected to work properly.

Finally, we remark that our method allows to generate strictly positive polynomials that are not sums of squares including a certificate without optimization, if \eqref{Equ:OurRelaxation} is satisfied for a polynomial $p \in \partial P_{3,6} \setminus \Sig_{3,6}$ (resp. for $(n,2d) = (4,4)$). Let $l_a$ be a separating extreme ray for $p$ and define $n := p + \lam \cdot r$ with $\lam \in \R_{> 0}$ and $r \in \text{int}(\Sig_{3,6})$ (resp. $r \in \text{int}(\Sig_{4,4})$). Then $n$ is strictly positive and evaluating $n$ on $l_a$ yields
\begin{eqnarray*}
	l_a(n)	& =	& l_a(p) + \lam \cdot l_a(r)
\end{eqnarray*}
with $l_a(p) < 0$ and $l_a(r) > 0$. Hence, we can immediately solve for $\lam$ such that $l_a(n) < 0$.

\section*{Acknowledgements}
We would like to thank Grigoriy Blekherman for his helpful advices. Furthermore, we would like to thank Cordian Riener and Thorsten Theobald for their support during the development of this article.

\bibliographystyle{amsplain}
\bibliography{BibPSDvsSOS}

\appendix
\section{An example for the seven point case}

As an example of our method we investigate the polynomial
\begin{eqnarray*}
	  p	& :=	& x^2 y^2 (x-y)^2+y^2 z^2 (y-z)^2+z^2 x^2 (z-x)^2+x y z (x-y) (y-z) (z-x),
\end{eqnarray*}
given by Reznick in \cite{Reznick:PSDandSOS}. It has the seven zeros
$$
\begin{array}{llll}
	  v_1 \ := \ (1,0,0), & v_2 \ := \ (0,1,0), & v_3 \ := \ (0,0,1), & v_4 \ := \ (1,1,0) \\
	  v_5 \ := \ (1,0,1), & v_6 \ := \ (0,1,1), & v_7 \ := \ (1,1,1). & \\
\end{array}
$$
We choose $v_8 := (-2,5,-1)$ which satisfies \eqref{Equ:GenCond}.
\begin{eqnarray*}
	  h	& :=	& b_1 x^3 + b_2 y^3 + b_3 z^3 + b_4 x^2 y + b_5 x^2 z + b_6 y^2 z + b_7 y^2 x + b_8 z^2 x + b_9 z^2 y + b_{10} x y z.
\end{eqnarray*}
We compute two ternary cubics
\begin{eqnarray*}
	q_1 \ = \ -x^2 y-35 x^2 z+y^2 x+35 z^2 x, \quad q_2 \ = \ 15 x^2 z-y^2 z-15 z^2 x+z^2 y
\end{eqnarray*}
vanishing on $v_1,\ldots,v_8$ by solving the system of linear equations
\begin{eqnarray*}
	h(v_1) \ = \ 0, \ldots, h(v_8) \ = \ 0
\end{eqnarray*}
on the coefficients of $h$. $q_1$ and $q_2$ are coprime since the genericity condition \eqref{Equ:GenCond} holds for $v_1,\ldots,v_8$ by Lemma \ref{Lem:ReznickBezout}. We obtain a ninth intersection point $v_9 = (3, 10, 1)$ of $q_1, q_2$ via computing the Gr\"obner basis 
\begin{eqnarray*}
	\{6z^4 - 11z^3 + 6z^2 - z, z^2y + 35z^3 - zy - 50z^2 + 15z, -y - 35z + y^2 + 35z^2, z - 1\}
\end{eqnarray*}
of $q_1$, $q_2$ and $z - 1$ with respect to lexicographic ordering. To compute the $u_j$ of the Cayley-Bacharach relation \eqref{Equ:CB} we solve the system of linear equations given by
\begin{eqnarray*}
	u_1 h(v_1) + \cdots + u_9 h(v_9) & =	& 0.
\end{eqnarray*}
We obtain $u = (84,-1260,-36,-90,63,35,-60,3,1)^T$ and therefore
\begin{eqnarray*}
	u_8^2 p(v_8) \ = \ 48456, \quad u_9^2 p(v_9) \ = \ 56016,
\end{eqnarray*}
i.e. \eqref{Equ:7PointEquality} holds and we find a separating hyperplane for $p$ on $A$. According to Corollary \ref{Cor:ExplicitRay} we choose $a_1,\ldots,a_7 := 10^9$ and $a_8 := 1$. By \eqref{Equ:a_9} we obtain
{\small
\begin{eqnarray*}
	a_9	& =	& \frac{-u_9^2}{\frac{u_1^2}{a_1}+\dots+\frac{u_8^2}{a_8}} \ = \ \frac{- 10^9}{84^2 + (-1260)^2 + (-36)^2 + (-90)^2 + 63^2 + 35^2 + (-60)^2 + 9 \cdot 10^9} \\
		& =	& -\frac{500000000}{4500806423}.
\end{eqnarray*}
}
We check correctness of our result by
\begin{eqnarray*}
	l_a(p)	& =	& a_1 p(v_1) + \cdots + a_9 p(v_9) \ = \ p(-2,5,-1) - \frac{500000000}{4500806423} \cdot p(3,10,1) \\
		& =	& -\frac{471957277321}{5626008028750} \ < \ 0.
\end{eqnarray*}
Hence, by Blekherman's Theorem \ref{Thm:BlekhermanExtremeRay}, $l_a$ is an extreme ray of $\Sig_{3,6}^*$ separating $p$ from the SOS cone. Since it is constructed in the way of Corollary \ref{Cor:ExplicitRationalRay}, this certificate is indeed rational.

\section{A non-appropriate point configuration for the Motzkin polynomial}
\label{App:Motzkin}
We choose the six zeros of the Motzkin polynomial and the instance $v_7 := (2/7,2/3,1)$ and $v_8 := (2/3,2/7,1)$.
These eight points satisfy the genericity condition \eqref{Equ:GenCond}. We proceed the same way as before and obtain two coprime ternary cubics
\begin{eqnarray*}
	q_1 & =	 & 272 z^3+520 x^2 y-461 x^2 z+189 y^2 z-520 z^2 y \\
	q_2 & =	 & -340 z^3-461 x^2 y+340 x^2 z-189 y^2 x+189 z^2 x+461 z^2 y
\end{eqnarray*}
vanishing on $v_1,\ldots,v_8$. The Gr\"obner basis for $q_1,q_2$ and $x - 1$ with respect to lexicographic ordering is
{\small
\begin{eqnarray*}
	& & \{1365 z-2014 z^2-425 z^3+1878 z^4-940 z^5+136 z^6, -329154 z+362620 z^2+266458 z^3 \\
	& & -257985 y-362620 z^4+62696 z^5+257985 z^2 y, -239454894 z+118408655 z^2+210552038 z^3 \\ 
	& & -167167820 z^4+28902856 z^5+48759165 y^2, x - 1\}
\end{eqnarray*}
}
and the ninth point is given by $v_9 = (1, 1, 65/34)$. The Cayley-Bacharach coefficients are
\begin{eqnarray*}
	u	& =	& \lf(\frac{-264}{7},\frac{-264}{7},\frac{891}{31},\frac{-99}{10},\frac{-99}{10},1,\frac{43659}{160},\frac{43659}{160},\frac{-4913}{62} \ri)^T
\end{eqnarray*}
But now
\begin{eqnarray*}
	(u_7^2 + u_8^2) \cdot (m(v_7) + m(v_8))	& = & 2 \cdot \frac{1906108281}{25600} \cdot 2 \cdot \frac{177025}{194481} \ = \ 271097.1914 \\
	u_9^2 \cdot m(v_9)			& = & 250270.0664.
\end{eqnarray*}
If we exchange the points $v_7$ and $v_9$ or the points $v_8$ and $v_9$ we obtain in both cases
\begin{eqnarray*}
	& & (u_7^2 + u_8^2) \cdot (m(v_7) + m(v_8)) \ = \ 3291366.873 \quad \text{ and } \quad u_9^2 \cdot m(v_9) \ = \ 67774.29785.
\end{eqnarray*}

\end{document}